\documentclass[preprint,12pt]{elsarticle}
\usepackage{mathtools,amsthm,amssymb}                                   
\usepackage{physics}                                                    
\usepackage{enumitem}                                                   
\usepackage{hyperref}

\newtheorem{theorem}{Theorem}
\newtheorem{conjecture}[theorem]{Conjecture}

\theoremstyle{definition}

\newtheorem{remark}[theorem]{Remark}

\everymath{\displaystyle}

\journal{Linear Algebra Appl.}

\begin{document}

\begin{frontmatter}
    \title{A matricial view of the Collatz conjecture}
    \author[1]{Pietro Paparella}
    \ead{pietrop@uw.edu}
    \ead[url]{https://faculty.washington.edu/pietrop/}
    \affiliation[1]{
        organization={Division of Engineering \& Mathematics},
        addressline={University of Washington Bothell}, 
        city={Bothell},
        postcode={98011}, 
        state={WA},
        country={U.S.}
    }

    \begin{abstract}
        In this note, it is shown that the nilpotency of submatrices of a certain class of adjacency matrices is equivalent to the aperiodic Collatz conjecture. 
    \end{abstract}
    

\end{frontmatter}

\section{Introduction}

Let $f:\mathbb{N} \longrightarrow \mathbb{N}$ be the function defined by 
\begin{align*}
    f(n) =
    \left\{ 
        \begin{array}{cc}
        \displaystyle\frac{n}{2},   & n \equiv 0 \bmod{2},  \\
        \displaystyle\frac{3n+1}{2}, & n \equiv 1 \bmod{2},
        \end{array}
    \right.
\end{align*}
and, for a fixed $n \in \mathbb{N}$, consider the sequence $s(n): \mathbb{N}_0 \longrightarrow\mathbb{N}$ defined by 
\[ 
s_k(n) \coloneqq 
\begin{cases} 
n, & k = 0, \\ 
f^k(n), & k > 0 
\end{cases}. \]
The function $f$ is called the \emph{(shortcut) Collatz function} and the sequence $s(n)$ is called the \emph{Collatz sequence of $n$}. 

It is well-known that if $n$ is a positive integer, then there are three possibilities for the sequence $s(n)$ \cite[Section 2]{lagarias1985}:

\begin{enumerate}
[leftmargin=2\parindent,label=(\roman*)]
    \item \emph{Convergent trajectory}. In this case, there is a positive integer $p$ such that $s_p(n) = f^p(n) = 1$. If $p$ is the smallest integer such that $s_p(n) = 1$, then $p$ is called the \emph{total stopping time (of $n$)}. 
    \item \emph{Nontrivial cycle}. In this case, there is a positive integer $n > 2$ such that $s_k(n) \ne 1$, $\forall k \in \mathbb{N}$ and the sequence $s(n)$ is \emph{periodic}---i.e., there are positive integers $p$ and $q$ such that $s_p(n) = f^p(n) = f^q(n) = s_q(n)$.
    \item \emph{Divergent sequence}. In this case, the sequence $s(n)$ is unbounded, i.e., $\lim_{k \to \infty} s_k(n) = \infty$.   
\end{enumerate}

As is well-known, the \emph{Collatz conjecture} (see, e.g., Lagarias \cite{lagarias1985} and references therein), asserts that all Collatz sequences are convergent. There are many known reformulations of this problem \cite{lagarias20113x1,lagarias20123x1}. The assertion that no nontrivial cyclic trajectories exist will be called the \emph{aperiodic Collatz conjecture}. 

Let $A: \mathbb{N} \times \mathbb{N} \longrightarrow \{ 0, 1 \}$ be the infinite matrix defined by
\begin{equation*}
    a_{ij} = \delta_{f(i),j} =
        \begin{cases}
        1, &  j = f(i),         \\
        0, & \text{otherwise}.
        \end{cases}    
\end{equation*} 
and for $n \in \mathbb{N}$, let $A_n$ be the $n$-by-$n$ matrix whose $(i,j)$-entry is $a_{ij}$, $1 \leqslant i,j \leqslant n$.

Introduced by Alves et al.~\cite{alvesetal2005}, the matrix $A_n$ is the adjacency matrix of the directed graph $\Gamma_n$ with vertices $V_n := \{1,\dots,n\}$ and (directed) edges $E_n := \{ (i,j) \in V_n \times V_n \mid j = f(i) \}$. In addition, they proved the following result. 

\begin{theorem}
[{\cite[Theorem 1 and Lemma 1]{alvesetal2005}}]
The aperiodic Collatz conjecture is equivalent to 
    \begin{equation}
    \label{alvescond}
        \trace(A_n^p) = 
            \begin{cases}
                2, & p \equiv 0 \bmod{2}, \\
                0, & p \equiv 1 \bmod{2}.
            \end{cases}
    \end{equation}
\end{theorem}
    
Since $f(2) = 1$ and $f(1) = 2$, it follows that
    \[ A_2 = \begin{bmatrix} 0 & 1 \\ 1 & 0 \end{bmatrix}, \]
and
    \begin{equation*}
        A_n =
        \begin{bmatrix}
        A_2 & 0                 \\
        B_n & C_n
        \end{bmatrix},    
    \end{equation*}
whenever $n > 2$. More formally, let $\mathbb{N}_3 \coloneqq \{ n \in \mathbb{N} \mid n \geqslant 3 \}$ and let $C: \mathbb{N}_3 \times \mathbb{N}_3 \longrightarrow \{0,1\}$ be the infinite matrix defined by $c_{ij} = a_{ij}$. For $n\geqslant 3$, let $C_n$ be the $(n-2)$-by-$(n-2)$ matrix whose $(i,j)$-entry is $c_{ij} = a_{ij}$, $3 \leqslant i,j \leqslant n$. For ease of notation, the rows and columns of $C_n$ are indexed by $\{3,\ldots, n\}$.

Denote by $\mathsf{M}_n (\mathbb{C})$ the algebra of $n$-by-$n$ matrices with complex entries. If $A \in \mathsf{M}_n (\mathbb{C})$, then $A$ is called \emph{nilpotent} if there is a positive integer $p$ such that $A^p = 0$. In such a case, the set $\{k\in \mathbb{N} \mid A^k = 0\}$ is nonempty and hence, by the well-ordering principle, contains a smallest element, which is called the \emph{index (of nilpotency) of $A$}.  

\begin{conjecture}
    \label{equivconj}
        If $n\geqslant 3$, then the matrix $C_n$ is nilpotent.
\end{conjecture}

The following result, which is the central aim of this work, reveals the import of Conjecture \ref{equivconj}.

\begin{theorem}
    \label{mainresult}
        The aperiodic Collatz conjecture is equivalent to Conjecture \ref{equivconj}.
\end{theorem}

\begin{remark}
If 
\[ M =
\begin{bmatrix}
A & 0 \\
B & C
\end{bmatrix} \in \mathsf{M}_n(\mathbb{C}), \]
then a straightforward proof by induction reveals that
\[ M^p =
\begin{bmatrix}
A^p                                                 & 0     \\
\displaystyle\sum_{k=0}^{p-1} C^k B A^{p - 1 - k}  & C^p
\end{bmatrix}, \forall p \in \mathbb{N}. \]
Consequently, $\trace(A_n^p) = \trace(A_2^p) + \trace(C_n^p)$ and Conjecture \ref{equivconj} implies \eqref{alvescond} since 
\begin{equation*}
\trace(A_2^p) = 
\begin{cases}
2, & p \equiv 0 \bmod{2}, \\
0, & p \equiv 1 \bmod{2}.
\end{cases}
\end{equation*}
\end{remark}

\section{Notation \& Background}

The following standard matrix-theoretic notation is adopted throughout. If $A \in \mathsf{M}_n(\mathbb{C})$, then 

\begin{itemize}
[leftmargin=\parindent]
    \item the $(i,j)$-entry of $A$ is denoted by $a_{ij}$, $a_{i,j}$, or $(A)_{ij}$; and
    \item $\trace(A) \coloneqq \sum_{k=1}^n a_{kk}$.
\end{itemize}

The following result lists several characterizations of nilpotency that are relevant to this work (for other characterizations, see, e.g., Horn and Johnson \cite{hj2013}). 

\begin{theorem}
\label{nilthm}
If $A \in \mathsf{M}_n(\mathbb{C})$, then the following are equivalent:
\begin{enumerate}
[leftmargin=2\parindent,label=\roman*)]
    \item $A$ is nilpotent;
    \item $A^n = 0$ \cite[p.~109]{hj2013}; and
    \item \label{tracezero} $\trace(A^p) = 0,\ \forall p \in \mathbb{N}$ \cite[Problem 2.4.P10]{hj2013}.
\end{enumerate}
\end{theorem}

If $V_n = \{ 1,\ldots, n \}$ and $E \subseteq V_n \times V_n$, then $\Gamma = (V_n,E)$ is called a \emph{directed graph} or \emph{digraph}. If $A = [a_{ij}]$ is the $n$-by-$n$ matrix such that 
\[ a_{ij} \coloneqq \begin{cases} 1, & (i,j) \in E, \\ 0, & (i,j) \notin E, \end{cases} \]
then $A$ is called the \emph{adjacency matrix of $\Gamma$}. If $W = \{ (i_0,i_1),\ldots,(i_{\ell-1},i_\ell) \} \subseteq E$, $\ell \geqslant 1$, then $W$ is called a \emph{(directed) walk (of length $\ell$)}. If $i_0 = i_\ell$, then $W$ is called a \emph{(directed) cycle (of length $\ell$)}. It is well-known that if $A \in \mathsf{M}_n(\mathbb{C})$ and $p \geqslant 2$, then 
\begin{equation*}
\label{ijentryofmatrixpower} 
(A^p)_{ij} = \sum_{k_1,\dots,k_{p-1} = 1}^n \left[ \prod_{\ell = 1}^p a_{k_{\ell-1},k_\ell} \right],    
\end{equation*} 
where $k_0 \coloneqq i$ and $k_p \coloneqq j$. In the context of adjacency matrices, the $(i,j)$-entry of $A^p$ yields the number of directed walks of length $p$ of $\Gamma$ beginning at node $i$ and terminating at node $j$. In particular, the $(i,i)$-entry of $A^p$ is the number of directed cycles of length $p$ that originate and terminate at node $i$.

\section{Proof of Main Result}

\begin{proof}[Proof of Theorem \ref{mainresult}]
Suppose that the aperiodic Collatz conjecture is true. If $n > 2$, then the subdigraph of $\Gamma_n$ with vertices $V_n\backslash \{1,2\}$ and edges 
$$E_n \backslash \{(1,2), (2,1), (4,2) \}$$ 
does not have directed cycles. Thus, every diagonal entry of $C_n^p$ is zero. Consequently, $\tr (C_n^p) = 0$, $\forall p \in \mathbb{N}$, and $C_n$ is nilpotent.

Conversely, suppose that $C_n$ is nilpotent $\forall n \in \mathbb{N}_3$ and, for contradiction, that the aperiodic Collatz conjecture is false. 

If there is a nontrivial cycle, then there is a positive integer $n > 2$ such that $s_k(n) \ne 1$, $\forall k \in \mathbb{N}$ and positive integers $p$ and $q$ such that $s_p(n) = f^p(n) = f^q(n) = s_q(n)$. Without loss of generality, it may be assumed that $p < q$. If $m \coloneqq s_p(n)$, then $s_{q-p}(m) = m$ and $\Gamma_m$ contains a cycle of length $q-p$. Consequently, the $(m-2,m-2)$-entry of $C_m^{q-p}$ is positive. As $C_m$ is nonnegative, it follows that $\trace(C_m^{q-p}) > 0$, which contradicts part \ref{tracezero} of Theorem \ref{nilthm}.
\end{proof}

\bibliographystyle{abbrv}
\bibliography{collatz}

\end{document}